\pdfoutput=1
\documentclass[a4paper,11pt]{article}
\usepackage[margin=2.5cm]{geometry}
\usepackage[T1]{fontenc}

\usepackage[english]{babel}
\usepackage{amsfonts, amsthm, amsmath, amssymb}
\usepackage{mathtools, tikz, tikz-cd, float}
\usepackage{hyperref}
\hypersetup{
colorlinks=true,
linkcolor=blue,
filecolor=magenta,      
urlcolor=cyan,
citecolor=[RGB]{0,204,0},
}
\usepackage[capitalise, nameinlink]{cleveref}
\usepackage{resizegather}
\usepackage[activate={true, nocompatibility}, final, tracking=true, kerning=true, spacing=true, factor=1100, stretch=10, shrink=10]{microtype}
\usepackage{authblk}
\usepackage{tikz}
\usepackage{dsfont}
\usetikzlibrary{matrix}

\usepackage{enumitem}
\setlist[enumerate,1]{label={(\roman{enumi})}, leftmargin=*}

\theoremstyle{definition}
\newtheorem{thm}{Theorem}[section]
\newtheorem{coroll}[thm]{Corollary}
\newtheorem{defn}[thm]{Definition}
\newtheorem{lemma}[thm]{Lemma}

\newtheorem{example}[thm]{Example}
\newtheorem{prop}[thm]{Proposition}

\newtheorem{remark}[thm]{Remark}

\newtheorem*{thm*}{Main Theorem}
\newtheorem{notn}[thm]{Notation}

\newcommand{\cA}{\mathcal{A}}

\newcommand{\cC}{\mathcal{C}}

\newcommand{\cM}{\mathcal{M}}

\newcommand{\cO}{\mathcal{O}}
\newcommand{\cQ}{\mathcal{Q}}
\newcommand{\cU}{\mathcal{U}}

\newcommand{\cL}{\mathcal{L}}

\newcommand{\bN}{\mathbb{N}}

\newcommand{\bZ}{\mathbb{Z}}

\newcommand{\Comod}{\mathrm{Comod}}
\newcommand{\fdcomod}{\mathrm{fdcomod}}
\newcommand{\fdmod}{\mathrm{fdmod}}
\newcommand{\reff}{^{\mathrm{eff}}}
\newcommand{\rfp}{^{\mathrm{fp}}}

\newcommand{\rop}{^{\mathrm{op}}}
\newcommand{\Vect}{\mathrm{Vect}}

\newcommand\blank{\underline{\ \ }}

\newcommand{\on}{\operatorname}

\newcommand{\Alg}{\on{Alg}}

\DeclareMathOperator{\colim}{colim}

\newcommand{\dash}{\text{-}}

\DeclareMathOperator{\Geff}{G\reff}
\DeclareMathOperator{\GL}{GL}

\DeclareMathOperator{\gr}{gr}
\DeclareMathOperator{\Gr}{Gr}
\DeclareMathOperator{\Gzero}{G_0}
\newcommand{\End}{\on{End}}
\newcommand{\Ext}{\on{Ext}}

\newcommand{\Hom}{\on{Hom}}
\DeclareMathOperator{\Ind}{Ind}
 
\DeclareMathOperator{\Keff}{K\reff}
\renewcommand{\ker}{\on{ker}}
\DeclareMathOperator{\Kzero}{K_0}

\DeclareMathOperator{\Mod}{Mod}
 
\DeclareMathOperator{\Perf}{Perf}

\newcommand{\Quot}{\on{Quot}}

\DeclareMathOperator{\rep}{rep}

\DeclareMathOperator{\RHom}{RHom}
\newcommand{\Spec}{\on{Spec}}

\DeclareMathOperator{\wt}{wt}

\newcommand{\Gm}{\mathbb{G}_{\mathrm{m}}}

\tikzset{
  symbol/.style={
    draw=none,
    every to/.append style={
      edge node={node [sloped, allow upside down, auto=false]{$#1$}}}
  }
}

\microtypecontext{spacing=nonfrench}

\begin{document}
\title{\textbf{Moduli of objects in finite length abelian categories}}
\author{Andres Fernandez Herrero, Emmett Lennen, and Svetlana Makarova}
\date{}

\maketitle
\abstract{We construct moduli spaces of objects in an abelian category satisfying some finiteness hypotheses. Our approach is based on the work of Artin-Zhang \cite{artin-zhang} and the intrinsic construction of moduli spaces for stacks developed by Alper-Halpern-Leistner-Heinloth \cite{alper2019existence}.}
\tableofcontents

\section{Introduction}
Moduli problems in algebraic geometry can roughly be divided into two classes: ``linear'' and ``nonlinear''.
Linear moduli problems classify objects that can be broken down to linear algebraic data, like representations, sheaves and complexes.
Nonlinear moduli problems often classify objects of more geometric nature, like certain classes of varieties, commutative or noncommutative.
In \cite{artin-zhang}, Artin and Zhang provide a framework to study the moduli of objects in a certain class of abelian categories, and they use it to construct Hilbert schemes of quotients of a noetherian object.
We view it as a general construction in the realm of linear moduli problems. The goal of this paper is to study the range of applicability of their theory in the context of finite length abelian categories.

We construct moduli spaces of objects in an abelian category $\cC$ satisfying the following conditions,
see \cref{section: abcats definitions} for definitions.
\begin{defn}
\label{standing assumptions on C}
    Let $k$ be a field. An essentially small $k$-linear abelian category $\cC$ is called \textit{nearly finite} if it satisfies the following:
    \begin{enumerate}[label=(A\arabic*), start=0]
        \item \label{standing assumption: C Hom-finite}
        is \emph{Hom-finite}, that is, Hom-spaces are finite dimensional;
        \item \label{standing assumption: C finite}
        is \emph{finite length}, that is, every object has finite length;
        \item \label{standing assumption: C has enough projectives}
        has enough projective objects.
    \end{enumerate}
\end{defn}

Our main results can be summarized as follows.
\begin{thm*}[\Cref{thm: good moduli space whole stack}, \Cref{thm: theta stratification and moduli space stability condition}] 
Let $\cC$ be a nearly finite category (\cref{standing assumptions on C}) and let $\cM_\cC$ be the stack of objects in $\cC$ as in \cref{definition: stack of objects in C}. Then,
\begin{enumerate}
    \item $\cM_\cC$ is a disjoint union of algebraic stacks of finite type over $k$.
    \item $\cM_\cC$ admits a good moduli space $M_\cC$ which is a disjoint union of local artinian $k$-schemes.
    \item K-theory classes naturally give rise to a notion of slope stability, which induces a $\Theta$-stratification on each connected component of $\cM_\cC$. The corresponding semistable locus admits a projective good moduli space.
\end{enumerate}
\end{thm*}

To illustrate the main theorem above, we point out that it refines King's construction of quiver moduli spaces when the quiver is acyclic or when the path algebra is finite-dimensional, which we now recall \cite{king}.
In that case the connected components of $\cM_{\cC}$ are in bijection with dimension vectors for the representations (more generally, in our setting we show that the connected components of $\cM_{\cC}$ are in natural bijection with effective G-theory classes in $\Geff(\cC)$ as in \Cref{subsection: K-theory and G-theory}). The K-theory classes are in natural correspondence with tuples of integers indexed by the vertices of the quiver. For each such tuple, King defined a notion of semistability for representations of quivers, and showed that the semistable locus admits a projective good moduli space over $k$.

Our proof of the main theorem uses the intrinsic criteria for the existence of moduli spaces developed in \cite{alper-good-moduli, alper2019existence}. Moreover, we follow the treatment of the moduli of abelian categories in \cite{alper2019existence} using the functor by Artin-Zhang. The main contribution of this paper is the observation that the finiteness conditions that we impose on our abelian categories guarantee that the moduli functor $\cM_{\cC}$ is represented by an algebraic stack which admits a good moduli space. 


In \Cref{section: example}, we provide several examples of abelian categories satisfying \Cref{standing assumptions on C}. 
In particular, our results provide new moduli spaces -- moduli of comodules over Hopf algebras of functions on quantum groups, which, to our knowledge, have not been constructed before, yet may be of interest to the representation theory community.

Even though our assumptions on the category $\cC$ seem quite restrictive, we believe that they are not very far from being sharp for the existence of a moduli space for the whole moduli functor $\cM_{\cC}$ in the sense of Artin-Zhang (before imposing any stability conditions). Indeed, Hom-finiteness \ref{standing assumption: C Hom-finite} is necessary in order for the functor to be represented by an algebraic stack. The finite length condition in \ref{standing assumption: C finite} is necessary for the functor $\cM_{\cC}$ to admit a good moduli space of finite type over $k$ by \cite[Lem. 7.19]{alper2019existence}. Finally, the assumption on the existence of enough projective objects \ref{standing assumption: C has enough projectives} is to our knowledge the only general categorical condition that guarantees that $\Ind(\cC)$ is adically complete as in \cite{artin-zhang}, which is a necessary condition for $\cM_{\cC}$ to be represented by an algebraic stack.

\subsection*{Related works}

King \cite{king} applies the methods of GIT to construct moduli of finite-dimensional representations of finite-dimensional associative algebras. These categories are known to have a categorical description, namely, they are nearly finite in the sense of \cref{standing assumptions on C} and have a finite number of simple objects up to isomorphism. 
Further, various stability conditions have been described on a subclass of those in \cite{futorny2007moduli}. 

Artin and Zhang \cite{artin-zhang} develop methods for defining families of objects in more general abelian categories, and prove that the Hilbert functor of quotients is representable by an algebraic space. This is a striking general result that they apply to constructing moduli spaces for noncommutative graded algebras. However, it relies on the assumption that the category is adically complete. This is guaranteed by the assumption that the category has enough projective objects, and it can also be checked directly for categories of coherent sheaves on projective schemes (commutative or noncommutative).

In \cite[Section 7]{alper2019existence}, the authors apply the theory of good moduli spaces to stacks of objects as defined in \cite{artin-zhang}. With a few technical assumptions, they prove existence of a proper good moduli space of Bridgeland semistable objects on the derived category of a proper smooth variety. Some of their more general results \cite[Thm. 7.27]{alper2019existence} assume as an input that the moduli stack is algebraic and locally of finite type.
Some technical work in \Cref{section: the stack of objects} is dedicated to verifying these assumptions for an abstract abelian category.

There have been constructions of moduli of objects in triangulated and dg categories. 
For example, To\"en and Vaqui\'e \cite{ToenVaquie-objects-in-dg-cats} construct derived moduli of perfect objects in dg categories, and Lieblich \cite{Lieblich-complexes-on-proper-morphism} constructs classical moduli of perfect universally gluable complexes on a proper morphism. While their constructions work in vast generality, their moduli stacks only see perfect objects. On the other hand, in our work, we also construct moduli of objects that are not necessarily perfect (for example, the moduli of modules over a finite $k$-algebra of infinite global dimension, such as $k[\epsilon]/(\epsilon^2)$).

\subsection*{Limitations}

Our framework leaves out interesting examples such as the moduli of coherent sheaves on a proper scheme or the moduli of objects in the heart of a Bridgeland stability condition. In these cases the whole moduli functor does not necessarily admit a good moduli space, but instead there is a semistability condition such that the semistable locus admits a good moduli space. 

The requirement that the category $\cC$ has enough projectives seems too restrictive to address some of these classical moduli problems (see \cite[Thm. 1.1]{Kanda} for a criterion of having enough projectives). Unfortunately, applying the theory of Artin and Zhang directly to arbitrary abelian categories without enough projectives does not usually yield functors represented by algebraic stacks (one can see this, for example, when we take $\cC$ to be the category of finite length modules over the polynomial algebra $k[t]$). It seems natural to attempt to modify the moduli functor, or to consider a natural embedding of the abelian category into one where the Artin-Zhang formalism is better behaved. This is part of ongoing investigations by the authors.

\subsection*{Vistas}

We view this project as the first step in our attempt to understand strange duality \cite{DT94, Beauville95, derksen-weyman,marian-oprea-strange} in the context of abelian categories. Exploring the strange duality phenomena for nearly finite categories will be part of our future research. 

A possible avenue of exploration for the interested reader would be to understand concretely the wall-crossing behavior of the stability conditions that we provide. It should be possible to use the $\Theta$-stratification and the linearity of the moduli problem to keep track of the variation of the moduli spaces in some of the examples in \Cref{section: example}. Furthermore, we have not attempted to investigate further the geometry of the moduli spaces that can be constructed using our theory, so this is another natural direction to investigate.

\subsection*{Notation}
Unless otherwise stated, all of our rings are assumed to be associative and unital, but not necessarily commutative. We shall work over a fixed algebraically closed ground field $k$. All of the rings, schemes and stacks are implicitly equipped with a morphism to $\Spec k$. For any two schemes (or stacks) $X$, $Y$ we write $X \times Y$ to denote the fiber product over $\Spec k$, any unadorned tensor product means tensoring over $k$, i.e., $\otimes = \otimes_k$. We denote by $k\dash\Alg$ the category of commutative $k$-algebras.

We will work with a fixed category $\cC$ that is nearly finite in the sense of \cref{standing assumptions on C}. Denote by $\cA = \Ind (\cC)$ the ind-completion of $\cC$. We will write $V$, $W$ for objects in abelian categories; $P$ and $S$ will usually denote projective and simple objects, respectively. We reserve letters $R$, $T$ to talk about $k$-algebras and $M$, $N$ for modules over them.

\subsection*{Acknowledgements}
We would like to thank Pieter Belmans, Pavel Etingof, Dan Halpern-Leistner, Andr\'es Ib\'a\~nez N\'u\~nez, Kimoi Kemboi, Andrew Kwon, Tony Pantev, Maximilien P\'eroux, Alekos Robotis and David Rydh for their interest in the project and helpful discussions.
The authors also thank the anonymous referees for their thoughtful comments and attention to detail. In particular, \Cref{proposition: MCalpha is connected and qc} (which greatly simplifies our original argument) and \Cref{remark 4.5} were suggested by a referee.

\section{Abelian categories and base-change}
\label{section: abcats definitions}

%
\subsection{Nearly finite categories}

A $k$-linear abelian category $\cC$ is \emph{Hom-finite} if for every $V, W \in \cC$, the dimension of $\Hom_\cC(V,W)$ over the base field $k$ is finite.
An object $V \in \cC$ is \emph{simple} if it has exactly two subobjects: $0$ and $V$.
The \emph{length} of an object $V \in \cC$ is the supremum of the numbers $n$ for which there exists a filtration
\[
    0=V_0 \subsetneq V_1 \subsetneq \dots \subsetneq V_n = V,
\]
possibly infinite.
Such a filtration is called a \emph{Jordan-H\"older filtration} if each factor $V_i / V_{i-1}$ is simple for $1\leq i \leq n$, and in this case we set $\gr V = \bigoplus V_i / V_{i-1}$. 

\begin{remark}
\label{remark: finite length iff noeth and art}
\label{remark: gr V is indep of JH filtrations}
An object has finite length if and only if it is both noetherian and artinian \cite[\href{https://stacks.math.columbia.edu/tag/0FCJ}{0FCJ}]{stacks-project}, and in this case $\gr V$ is independent, up to isomorphism, of the choice of the Jordan-H\"older filtration \cite[\href{https://stacks.math.columbia.edu/tag/0FCK}{0FCK}]{stacks-project}.
\end{remark}

\begin{lemma} 
\label{lemma: projective covers of simples}
Let $\cC$ be a nearly finite abelian category (\cref{standing assumptions on C}).
For every simple object $S \in \cC$, there is a unique (up to isomorphism) indecomposable projective object $P_S \in \cC$ that admits an epimorphism $P_S \twoheadrightarrow S$. If $S'$ is another simple object that is not isomorphic to $S$, then $\Hom(P_S, S') =0$.
\end{lemma}

\begin{proof}
Let $S$ be a simple object. By \cref{standing assumptions on C}, there is a projective object $P$ in $\cC$ and a nonzero morphism $P \to S$, which is necessarily an epimorphism since $S$ is simple.
By the finite length assumption in \cref{standing assumptions on C}, $P$ admits a finite direct sum decomposition $P = \bigoplus_j P_j$, where each $P_j$ is indecomposable projective. Since $P_j \to S$ is nonzero for some $j$, we can set $P_S = P_j$. Uniqueness follows from sequentially applying \cite[Thm. 5.5, Lem. 3.6, Cor. 3.5]{krause-ks-cats}.

Let $S'$ be any other simple object in $\cC$, and suppose that there is a nonzero morphism $P_S \to S'$, which necessarily is an epimorphism. Since both $S'$ and $S$ are simple, the kernels of $P_S \to S$ and $P_S \to S'$ are maximal proper subobjects of $P_S$, and so they must agree by \cite[Lem. 3.6]{krause-ks-cats}. This implies $S \cong S'$.
\end{proof}

\begin{notn}
\label{notation: I and S_i and P_i}
\label{defn: canonical set projective generators}
    Let $I$ denote the set of isomorphism classes of simple objects in $\cC$. We denote by $\{ S_i \}_{i \in I}$ a set of simple objects from each of the isomorphism classes, and choose their respective indecomposable projective covers $\{ P_i \}_{i \in I}$.
    By \cref{lemma: projective covers of simples}, the latter set is independent of any choice, up to isomorphism.
    We call $\{P_i\}_{i \in I}$ as above the \emph{canonical set of projective generators}.
\end{notn}

\subsection{K-theory and G-theory}
\label{subsection: K-theory and G-theory}

In this subsection, we recall the standard Grothendieck group construction. In algebraic geometry, this construction is often applied to the category of coherent sheaves on a smooth variety, which happens to coincide with the category of perfect objects. In general however, because abelian categories may not have finite global dimension, we need to distinguish Grothendieck groups of the given abelian category and of the subcategory of perfect objects. We mimic the construction of Quillen’s algebraic K-theory spectrum \cite{quillen-K-theory} as a group completion, with respect to direct sum as the group operation, of the groupoid of finitely generated projective modules, omitting the higher homotopy data.

\begin{defn}
\label{definition: perfect objects}
Let $\cC$ be an essentially small abelian category.
\begin{itemize}
    \item An object $V$ of $\cC$ is \emph{perfect} if it admits a finite resolution by projective objects in $\cC$. We denote by $\Perf (\cC)$ the full additive subcategory of $\cC$ consisting of perfect objects.
    \item The \emph{(zeroth) K-theory} of $\cC$ is denoted by $\Kzero (\cC)$ and is the abelian group generated by isomorphism classes of objects in $\Perf (\cC)$, subject to the additivity relations in short exact sequences.
    \item 
    The \emph{(zeroth) G-theory} of $\cC$ is denoted by $\Gzero (\cC)$ and is the abelian group generated by isomorphism classes of objects in $\cC$, subject to the additivity relations in short exact sequences.
    \item 
    We say that a G-theory class $\alpha \in \Gzero (\cC)$ is \emph{effective} if there is $V \in \cC$ such that $\alpha = {[V]}$. We denote by $\Geff (\cC)$ the submonoid of effective G-theory classes. We similarly define $\Keff (\cC)$ as the set of $\alpha \in \Kzero(\cC)$ such that $\alpha = [V]$ for some perfect object $V$.
\end{itemize}
\end{defn}

We will prove in \cref{lemma: dual bases} that we have a natural embedding $\iota \colon \Kzero(\cC) \to \Gzero(\cC)$.
We note, however, that $\Keff(\cC) \subseteq \Kzero(\cC) \cap \Geff(\cC)$ is not an equality in general.
For example, if $\cC$ is the category of finite-dimensional representations of quiver $A_1 = \begin{tikzcd} \bullet \arrow[r] & \bullet \end{tikzcd}$, then there are two projective objects $P_1 = (k \to k)$ and $P_2 = (0\to k)$.
The object $S_1 = (k\to 0)$ is perfect, because it admits a two-step projective resolution $0 \to P_2 \to P_1 \to S_1 \to 0$; however, its K-theory class is not effective.

\begin{lemma}
\label{lemma: define Euler pairing}
Suppose that $\cC$ is nearly finite. Let $V$ be an object of $\Perf (\cC)$ and $W$ be an object of $\cC$, then we have the following.
\begin{enumerate}
    \item $\Ext^i(V,W) = 0$ for all $i \gg 0$.
    \item For all $i$, we have $\dim_k \Ext^i(V,W) < \infty$.
    \item There is a natural pairing $\Kzero(\cC) \otimes_{\bZ} \Gzero(\cC) \to \bZ$ defined on effective classes as
    \[
        ([V], [W]) \longmapsto \langle [V], [W] \rangle = \chi \left( \RHom_\cC (V, W) \right).
\]
    We write $\langle V, W \rangle = \langle [V], [W] \rangle$ for brevity. \qed
\end{enumerate}
\end{lemma}

\begin{lemma}
\label{lemma: dual bases}
Suppose that $\cC$ is nearly finite. Let $\{P_i\}_{i \in I}$ be the canonical set of projective generators as in \cref{defn: canonical set projective generators}.
    \begin{enumerate}
        \item \label{item: dual bases}
        We have $\Kzero (\cC) \cong \bZ^{\oplus I}$ with basis $\{ [P_i] \}_{i \in I}$, and $\Gzero (\cC) \cong \bZ^{\oplus I}$ with basis $\{ [S_i] \}_{i \in I}$. These bases are dual up to positive scalar multiples, i.e., $\langle P_i, S_i \rangle = \tau_i \geq 1$ and $\langle P_i, S_j \rangle = \delta_{ij} \tau_i$.
        \item \label{item: Euler pairing is nondegenerate}
        The pairing $\Kzero (\cC) \times \Gzero (\cC) \to \bZ$ from \cref{lemma: define Euler pairing} is nondegenerate.
        \item \label{item: injection K into G} \label{item: if C=Perf C then K=G}
        There is a natural injection $\iota \colon \Kzero (\cC) \to \Gzero (\cC)$ given by $\iota([V]) = [V]$.
        Furthermore, if $\cC = \Perf (\cC)$, then $\iota$ is an isomorphism.
    \end{enumerate}
\end{lemma}

\begin{proof}
    For Part \ref{item: dual bases}, the isomorphism $\Kzero(\cC) \cong \bZ^{\oplus I}$ follows from the fact that every short exact sequence of projective objects splits and the Krull-Schmidt property \cite[Lem. 5.1, Thm. 5.5, Thm. 4.2]{krause-ks-cats}.
    On the other hand, the isomorphism for $\Gzero(\cC)$ follows from the existence of Jordan-H\"older filtrations and the uniqueness of the corresponding associated graded objects.
    The fact that they are dual bases up to scaling follows from the vanishing of $\Hom(P_S, S')$ stated in \Cref{lemma: projective covers of simples}.
    Part \ref{item: Euler pairing is nondegenerate} follows at once from the existence of dual bases.

    Injectivity in Part~\ref{item: if C=Perf C then K=G} follows from the Jordan-H\"older property and from the fact that every short exact sequence of projective objects splits. 
    Part~\ref{item: if C=Perf C then K=G} is a direct consequence of \cref{definition: perfect objects}.
\end{proof}

\subsection{Finiteness conditions in Grothendieck categories}

In this subsection, we fix a Grothendieck category $\cA$. Recall that an abelian category is called \emph{Grothendieck} if it has arbitrary direct sums, filtered colimits are exact, and it has a generator. A \emph{generator} is an object $G \in \cA$ such that $\Hom_\cA(G,\blank)$ detects zero morphisms.

\begin{defn}[{\cite[pp. 91, 92]{popescu}}]
    An object $V \in \cA$ is called
    \begin{itemize}
        \item \emph{finitely generated (finite type)} if,
        whenever $V \cong \colim_i V_i$ for a filtered diagram
        of subobjects $V_i \subset V$,
        there is an index $i$ for which $V = V_i$;
        \item \emph{finitely presented} if it is finitely generated, and 
        the kernel of any epimorphism $W \twoheadrightarrow V$ from a finitely generated object $W$ is finitely generated itself; equivalently, if $\Hom_\cA(V,\blank)$ commutes with filtered colimits;
        \item \emph{noetherian} if every subobject of $V$ is finitely generated; equivalently, if subobjects of $V$ satisfy the ascending chain condition.
    \end{itemize}
\end{defn}

When $\cA \cong R\dash\Mod$, the three notions are equivalent to the algebraic notions of finite generation, finite presentation and being noetherian, respectively. 
In a more general category admitting filtered colimits, objects $V$ for which $\Hom_{\cA}(V,\blank)$ commutes with filtered colimits are called \emph{compact}. These have a special importance for the moduli theory we consider. 

\begin{notn}
    We denote by $\cA\rfp$ the full subcategory of finitely presented objects in $\cA$. 
\end{notn}

\begin{defn}
    A Grothendieck category is called \emph{locally noetherian} if it admits a set of noetherian projective generators.
\end{defn}

We recall the following useful fact.

\begin{thm}[{\cite[Ch. 5, Thm. 8.7]{popescu}}]
    Let $\cA$ be a locally noetherian category, then the three notions of finiteness for objects coincide. Hence, $\cA\rfp$ is an abelian category whose kernels and cokernels coincide with those in $\cA$.
\end{thm}

\subsection{Base-change of abelian categories}
\label{subsection: background base-change}

Let $\cA$ be a $k$-linear Grothendieck abelian category and $R$ a commutative $k$-algebra. We recall some notions from \cite{artin-zhang}.

\begin{defn}[{\cite[B2]{artin-zhang}}] We denote by $\cA_R$ the category of pairs $(V, \psi)$ where $V \in \cA$ and $\psi$ is a homomorphism of $k$-algebras $\psi \colon R \to \End_{\cA}(V)$.
Objects of $\cA_R$ are called \emph{$R$-module objects} in $\cA$, and the homomorphism $\psi$ is called an \emph{$R$-module structure} on $V$.
The set of morphisms between two pairs $(V, \psi) \to (W, \xi)$ consists of those morphisms in $\Hom_{\cA}(V, W)$ that are compatible with the corresponding $R$-module structures on $V$ and $W$.
\end{defn}

For any two commutative $k$-algebras $R$, $T$ with a homomorphism $R \to T$, there is an associated forgetful functor $|_R \colon \cA_T \to \cA_R$.

In \cite[B3]{artin-zhang}, Artin and Zhang define a right exact functor of abelian categories $V\otimes_R \blank \colon R\dash\Mod \to \cA_R$ for an object $V$ of $\cA_R$. For any fixed $R$-module $M$, they similarly define a right exact functor $\blank\otimes_R M \colon \cA_R \to \cA_R$. If $M=T$ is equipped with the structure of a $k$-algebra, then $\blank \otimes_R T$ factors through the forgetful functor $\cA_T \to \cA_R$, and so we get a base-change functor $\blank\otimes_R T \colon \cA_R \to \cA_T$.

\begin{prop}[{\cite[B3.16]{artin-zhang}}]
\label{prop: tensor-Hom adjunction}
    Let $R \to T$ be a map of $k$-algebras. Then the functor 
    $\blank \otimes_R T \colon \cA_R \to \cA_T$ is the left adjoint of
    $|_R \colon \cA_T \to \cA_R$. \qed
\end{prop}

\begin{defn}[{\cite[C1]{artin-zhang}}] \label{defn: flatness} 
An object $V \in \cA_R$ is called \emph{$R$-flat} if the functor $V \otimes_R \blank \colon R\dash\Mod \to \cA_R$ is exact.
\end{defn}

For the current setup of moduli theory to work, we need the following condition.

\begin{defn}[see \cite{artin-small-zhang}] \label{defn: strongly noetherian} The category $\cA$ is called \emph{strongly locally noetherian} if for all noetherian $k$-algebras $R$, the base-change $\cA_R$ is locally noetherian.
\end{defn}

It is nontrivial to check \Cref{defn: strongly noetherian} in general \cite{resco-small}. We prove that under our assumptions, the strong noetherian property is automatic.

Let $\cC$ be a $k$-linear abelian category which is Hom-finite \ref{standing assumption: C Hom-finite} and finite length \ref{standing assumption: C finite}. Denote by $\cA$ its \emph{ind-completion} $\cA = \Ind (\cC)$ \cite[\S8.6]{kashiwara-schapira}.
By \cite[Thm. 8.6.5(vi)]{kashiwara-schapira}, $\cA$ is Grothendieck.
Every finitely presented object $V$ in $\cA$ can be shown to be a direct summand of an object from $\cC$, because $V$ is a filtered colimit of some objects $C_j \in \cC$, and then the identity $\operatorname{id}_V$ in $\Hom(V, \colim_j C_j)$ factors through some $C_\ell$.
By \cite[Prop. 8.6.11]{kashiwara-schapira}, the subcategory $\cC$ is thick in $\cA$, hence
we have $\cA\rfp \simeq \cC$.

\begin{prop}
\label{proposition: Hom-finite and finite cat is fdcomod}
    If $\cC$ is Hom-finite \ref{standing assumption: C Hom-finite} and finite length \ref{standing assumption: C finite}, then $\cC$ is equivalent to the category $H \dash \fdcomod$ 
    of finite dimensional comodules over some $k$-coalgebra $H$.
\end{prop}

\begin{proof}
    Our assumptions imply that $\cA = \Ind (\cC)$ is of finite type in the sense of \cite[Def. 4.4]{Takeuchi}, because $\cA$ is Grothendieck as noted above, and the other conditions are satisfied by \ref{standing assumption: C Hom-finite} and \ref{standing assumption: C finite}.
    By \cite[Thm. 5.1]{Takeuchi}, $\cA$ is equivalent to the category $H \dash \Comod$ of comodules over some coalgebra $H$.
    This identifies $\cC$ with a subcategory of $H \dash \Comod$ of objects of finite length.
    The fundamental theorem of coalgebra states that every $H$-comodule is a union of finite dimensional subcomodules \cite[Thm. 2.1.7]{DNR_Hopf_algebras}, so a comodule of finite length is necessarily finite dimensional.
\end{proof}

\begin{thm} \label{thm: ind-completions of Hom-finite and finite categories are strongly locally noetherian}
    If $\cC$ is an essentially small $k$-linear abelian category which is Hom-finite \ref{standing assumption: C Hom-finite} and finite length \ref{standing assumption: C finite}, then $\cA = \Ind (\cC)$ is strongly locally noetherian.
\end{thm}
Since $\cA \simeq \Ind(\cA\rfp)$, this automatically implies that $\cA$ is strongly locally noetherian whenever it is a locally noetherian Grothendieck category for which $\cA\rfp$ is Hom-finite \ref{standing assumption: C Hom-finite} and finite length \ref{standing assumption: C finite}.

\begin{proof}
    Since $\cC$ is a finite length category, it is noetherian, hence $\cA$ is locally noetherian. 
    
    Let $R$ be a noetherian $k$-algebra.
    By \cite[B3.17]{artin-zhang}, $\cC \otimes R = \{ V \otimes R \mid V \in \cC \}$ generates $\cA_R$. If we show that for every $V \in \cC$, its base change $V \otimes R$ is noetherian, then the claim will be proved. So pick $V \in \cC$
    and consider the following commutative diagram of functors, which exists by \cref{proposition: Hom-finite and finite cat is fdcomod}, where the vertical arrows are the forgetful functors.
\[\begin{tikzcd}
	{(H \dash \Comod)_R} && {H \dash \Comod} \\
	\\
	{R \dash \Mod} && {\Vect_k}
	\arrow["{R\otimes \blank}"', from=1-3, to=1-1]
	\arrow["F", from=1-3, to=3-3]
	\arrow["{F_R}"', from=1-1, to=3-1]
	\arrow["{R\otimes \blank}"', from=3-3, to=3-1]
\end{tikzcd}\]
    By \cref{proposition: Hom-finite and finite cat is fdcomod}, $ F(V) \otimes R$ is a finite free module over a noetherian ring $R$, hence itself noetherian. But $F_R (V \otimes R) = F(V) \otimes R$, so it forces $V \otimes R$ to be noetherian.
\end{proof}

\begin{coroll}
\label{corollary: base change to fields yields hom-finite and finite categories}
    If $\cC$ is an essentially small $k$-linear abelian category which is Hom-finite \ref{standing assumption: C Hom-finite} and finite length \ref{standing assumption: C finite}, then for any field extension $k \subset K$, finitely presented objects in the base change category $((\Ind \cC)_K)\rfp$ form a Hom-finite and finite length category as well.
\end{coroll}

\begin{proof}
    This follows by specializing the proof of \cref{thm: ind-completions of Hom-finite and finite categories are strongly locally noetherian} to the case of $R=K$. Indeed, we get 
    $(\Ind \cC)_K \cong (H \dash \Comod)_K \cong (H_K) \dash \Comod$, so the subcategory of finitely presented objects consists of comodules that are finite dimensional over $K$.
\end{proof}

\begin{example}[Necessity of Hom-finiteness]
    \Cref{thm: ind-completions of Hom-finite and finite categories are strongly locally noetherian} does not hold if we remove the assumption that $\cC$ is Hom-finite. As an example, let $\cC$ denote the category of finite dimensional vector spaces over the algebraic closure $\overline{k(t)}$ of the field of rational functions $k(t)$. We can view $\cC$ as a $k$-linear category which is not Hom-finite. Every element in $\cC$ has finite length. The ind-completion $\cA$ is the category of all vector spaces over $\overline{k(t)}$. We set $R = \overline{k(t)}$. Then the base-change $\cA_{\overline{k(t)}}$ is the abelian category of $\overline{k(t)} \otimes_k \overline{k(t)}$-modules. The trivial rank 1 module $\overline{k(t)} \otimes_k \overline{k(t)}$ is finitely presented, but it is not noetherian. Indeed, we have a surjection $\overline{k(t)} \otimes_k \overline{k(t)} \to \overline{k(t)} \otimes_{k(t)} \overline{k(t)}$ and the algebra $\overline{k(t)} \otimes_{k(t)} \overline{k(t)}$ has infinitely many minimal primes corresponding to elements of the Galois group $Gal(\overline{k(t)}/k(t))$. This means that $\cA$ is not strongly noetherian.
\end{example}

We end this subsection by stating two useful lemmas about the base-change to a noetherian algebra $R$. As usual, we use the notation $\cA = \Ind(\cC)$.
\begin{lemma}[Nakayama's lemma {\cite[Thm. C4.3]{artin-zhang}}] \label{lemma:nakayama abelian cats} 
Suppose that $\cC$ satisfies \ref{standing assumption: C Hom-finite} and \ref{standing assumption: C finite}. Let $R$ be a noetherian $k$-algebra, and let $W \in (\cA_R)\rfp$ be a noetherian object. Then there is an open subscheme $U \subset \Spec(R)$ such that for any $R$-algebra $T$ we have $W \otimes_R T =0$ if and only if $\Spec(T) \to \Spec(R)$ factors through $U$.
\end{lemma}

\begin{lemma} 
\label{lemma:generating set base-change}
Suppose that $\cC$ satisfies  \Cref{standing assumptions on C}, and let $\{P_i \}_{i \in I}$ be a set of projective generators in $\cC$. For any noetherian $k$-algebra $R$, we have that $\{P_i \otimes R\}_{i \in I}$ is a set of noetherian projective generators in $\cA_{R} = (\Ind \cC)_R$.
\end{lemma}

\begin{proof}
The set $\{P_i \otimes R\}_{i \in I}$ generates $\cA_{R}$ by \cite[Cor. B3.17]{artin-zhang}. Each $\{P_i \otimes R\}_{i \in I}$ is projective by 
tensor-Hom adjunction (\cref{prop: tensor-Hom adjunction}),
or by \cite[Lem. D3.2]{artin-zhang}.
Each $P_i \otimes R$ is noetherian by \Cref{thm: ind-completions of Hom-finite and finite categories are strongly locally noetherian}.
\end{proof}

\subsection{Behavior of \texorpdfstring{$\Hom$}{Hom} in families}

\begin{lemma} \label{lemma:homs from projectives families}
Let $\cA$ be a Grothendieck $k$-linear category. Let $R$ be a $k$-algebra. 
Let $P \in (\cA_{R})\rfp$ be a finitely generated projective object and let $V \in \cA_R$ be any object. 
Then we have the following.
\begin{enumerate}
    \item \label{item: Hom from fp proj commutes with tensor}
    For all $R$-modules $M$, the natural morphism 
    \[\Hom_{\cA_{R}}(P, V) \otimes_{R} M \to \Hom_{\cA_{R}}(P, V \otimes_{R} M) \]
    is an isomorphism.
    \item \label{item: Hom from fp proj commutes with base change}
    If $T$ is an $R$-algebra, then the natural morphism
    \[ \Hom_{\cA_{R}}(P, V)\otimes _{R} T \to \Hom_{\cA_{T}}(P\otimes_{R} T, V \otimes_{R} T)\]
    is an isomorphism.
    \item \label{item: Hom from fp proj preserves flatness}
    If $V \in \cA_R$ is $R$-flat, then $\Hom_{\cA_{R}}(P,V)$ is an $R$-flat module.
\end{enumerate}
\end{lemma}

\begin{proof}
    For part \ref{item: Hom from fp proj commutes with tensor}, choose a presentation $R^{\oplus J} \to R^{\oplus L} \to M \to 0$ of the module $M$. This induces a commutative diagram below.
\[
    \begin{tikzcd}
      \Hom_{\cA_{R}}(P, V)^{\oplus J} \ar[r] \ar[d] & \Hom_{\cA_{R}}(P, V)^{\oplus L} \ar[d] \ar[r] & \Hom_{\cA_{R}}(P, V)\otimes_{R} M \ar[d] \ar[r] & 0 \\ \Hom_{\cA_{R}}(P, V^{\oplus J})  \ar[r]  & \Hom_{\cA_{R}}(P, V^{\oplus L}) \ar[r] & \Hom_{\cA_{R}}(P, V \otimes_{R} M) \ar[r] & 0.
    \end{tikzcd}
\]
    Both rows are exact, because tensor products are right exact and $P$ is projective. Since $P$ is finitely presented, 
    the two leftmost vertical morphisms are isomorphisms, and we conclude by the five lemma that $\Hom_{\cA_{R}}(P, A) \otimes_{R} M \to \Hom_{\cA_{R}}(P, A \otimes_{R} M)$ is an isomorphism.

    This now implies part \ref{item: Hom from fp proj commutes with base change}, because we have natural morphisms
    \[ \Hom_{\cA_{R}}(P, V)\otimes _{R} T \to \Hom_{\cA_{R}}(P, V\otimes_{R} T) \to \Hom_{\cA_{T}}(P\otimes_{R}{T}, V \otimes_{R}{T}) , \]
    and the first one is an isomorphism by part 
    \ref{item: Hom from fp proj commutes with tensor},
    while the second one is an isomorphism by tensor-Hom adjunction (\cref{prop: tensor-Hom adjunction}).

    To prove part \ref{item: Hom from fp proj preserves flatness}, we choose an injective morphism $N \hookrightarrow M$ of $R$-modules and check that it remains injective after tensoring with $\Hom_{\cA_{R}}(P,V)$. We have a commutative diagram.
 \[
    \begin{tikzcd}
      \Hom_{\cA_{R}}(P, V)\otimes_{R} N \ar[r] \ar[d] & \Hom_{\cA_{R}}(P, V)\otimes_{R}M \ar[d] \\ 
      \Hom_{\cA_{R}}(P, V\otimes_{R} N)  \ar[r]  & \Hom_{\cA_{R}}(P, V\otimes_{R} M).
    \end{tikzcd}
\]
    By part \ref{item: Hom from fp proj commutes with tensor}, both vertical morphisms are isomorphisms. Since $V$ is $R$-flat, the morphism $V \otimes_{R} N \to V \otimes_{R} M$ is a monomorphism. By left exactness of $\Hom_{\cA_{R}}(\blank,\blank)$, we conclude that the bottom morphism is injective. Therefore, $\Hom_{\cA_{R}}(P, V)\otimes_{R} N  \to \Hom_{\cA_{R}}(P, V)\otimes_{R}M$ is injective.
\end{proof}

\begin{lemma} \label{lemma: homs from fp proj preserve finiteness}
Assume $\cA = \Ind (\cC)$, where $\cC$ satisfies \cref{standing assumptions on C}. Let $R$ be a $k$-algebra,
and let $P \in \cA \rfp$ be a finitely presented projective object. 
If $V \in \cA_R$ is finitely presented, then $\Hom_{\cA_{R}}(P \otimes R,V)$ is a finitely presented $R$-module.
\end{lemma}
\begin{proof}
    By \cref{lemma:generating set base-change} $\{P_i\otimes R\}_{i \in I}$ is a set of noetherian projective generators in $\cA_{R}$. 
    Since $V$ is finitely generated, there exists an epimorphism $\bigoplus_\ell P_\ell\otimes R \twoheadrightarrow V$ for a finite set of indices $\ell$, with each $P_\ell \in \cC$ projective. Since $V$ is finitely presented, the kernel $K$ of $\bigoplus_\ell P_\ell \otimes R \twoheadrightarrow V$ is finitely generated, and so it also admits a surjection $\bigoplus_j P_j' \otimes R$ for another finite tuple of projective objects $P'_j \in \cC$. Using that $P \otimes R$ is projective, we can apply the exact functor $\Hom_{\cA_R}(P \otimes R, \blank )$ to the presentation $\bigoplus_j P_j' \otimes R \to \bigoplus_{\ell} P_{\ell} \otimes R \to V$ in order to get an exact sequence
    \[\bigoplus_j \Hom_{\cA_R}(P \otimes R, P_j' \otimes R) \to \bigoplus_{\ell} \Hom_{\cA_R}(P \otimes R, P_{\ell} \otimes R) \to \Hom_{\cA_R}(P \otimes R, V) \to 0. \]
    By \Cref{lemma:homs from projectives families} \ref{item: Hom from fp proj commutes with base change}, we have that 
    \[\Hom_{\cA_{R}}(P\otimes R, P_j'\otimes R) \cong \Hom_{\cA}(P, P_j') \otimes R \; \; \text{ and} \;\;\Hom_{\cA_{R}}(P\otimes R, P_{\ell} \otimes R) \cong \Hom_{\cA}(P, P_{\ell}) \otimes R\]
    are finite free $R$-modules (recall that $\Hom_{\cA}(P, P_{\ell})$ and $\Hom_{\cA}(P, P_j')$ are finite $k$-vector spaces). Therefore, the exact sequence above shows that $\Hom_{\cA_{R}}(P \otimes R,V)$ is a finitely presented $R$-module.
\end{proof}

\section{The stack of objects in \texorpdfstring{$\cC$}{C}} \label{section: the stack of objects}

In this section, we use the notation $\cA = \Ind (\cC)$, where $\cC$ is nearly finite (\cref{standing assumptions on C}).
In particular, $\cA$ is strongly locally noetherian by \Cref{thm: ind-completions of Hom-finite and finite categories are strongly locally noetherian}. We fix once and for all the canonical set of projective generators $\{P_i\}_{i \in I}$ as in \Cref{defn: canonical set projective generators}.

\begin{defn}
\label{definition: stack of objects in C}
We denote by $\cM_{\cC}$ the pseudofunctor from $k\dash\Alg$ into groupoids that sends a $k$-algebra $R$ into the groupoid of $R$-flat objects of finite presentation in the abelian category $\cA_{R}$.
\end{defn}

\begin{defn}
Let $F$ denote an object in $\cA\rfp \simeq \cC$. We denote by $\Quot(F)$ the functor from $k$-$\Alg$ into sets that sends a $k$-algebra $R$ into the set of equivalence classes of quotients $F\otimes R \twoheadrightarrow V$ in $\cA_{R}$, where $V$ is $R$-flat and of finite presentation. Two such quotients are considered equivalent if they have the same kernel.
\end{defn}

\begin{prop}
For any $F \in \cC$, the functor $\Quot(F)$ is represented by a separated algebraic space locally of finite type over $k$. Moreover, it satisfies the valuative criterion of properness.
\end{prop}
\begin{proof}
It suffices to show that the $k$-linear abelian category $\cA$ satisfies the hypotheses of \cite[Thm. E3.1]{artin-zhang}. The category $\cA$ is strongly locally noetherian by \cref{thm: ind-completions of Hom-finite and finite categories are strongly locally noetherian}, and adically complete by 
\cite[Cor. D3.3]{artin-zhang}. It is Ext-finite by the existence of projective resolutions by objects in $\cA\rfp = \cC$ and \Cref{lemma: homs from fp proj preserve finiteness}. The other assumptions needed are implied directly by \cref{standing assumptions on C}. The valuative criterion for properness follows from \cite[Lem. E3.3]{artin-zhang}.
\end{proof} 

\begin{defn}
For all $F \in \cC$, we let $u_{F}: \Quot(F) \to \cM_{\cC}$ denote the morphism of functors that for every $k$-algebra $R$ sends an equivalence class of quotients $[q \colon F \otimes R \twoheadrightarrow V]$ to the object $u_{F}([q]) := F\otimes R / \ker(q)$.
\end{defn}

Recall that $I$ was the indexing set for canonical projective generators $\{P_i\}_{i \in I}$.
\begin{notn}
For every tuple of nonnegative integers $\vec{n} = (n_i)_{i \in J}$ indexed by a finite subset $J \subset I$, we set $P_{\vec{n}} = \bigoplus_{i \in J} P_i^{\oplus n_i}$
\end{notn}

Consider the morphism 
$\bigsqcup_{\vec{n}}u_{P_{\vec{n}}} \colon \bigsqcup_{\vec{n}} \Quot(P_{\vec{n}}) \to \cM_{\cC}$, 
where the disjoint union runs over all tuples $\vec{n}$ of nonnegative integers indexed by finite subsets $J \subset I$ as above. 

\begin{lemma} \label{lemma:quot atlas}
The morphism of functors $\bigsqcup_{\vec{n}}u_{P_{\vec{n}}} \colon \bigsqcup_{\vec{n}} \Quot(P_{\vec{n}}) \to \cM_{\cC}$ is sche\-ma\-tic, smooth and surjective. Moreover, for each fixed $\vec{n}$ the morphism $u_{P_{\vec{n}}}\colon \Quot(P_{\vec{n}}) \to \cM_{\cC}$ is quasi-affine.
\end{lemma}

Before we prove this lemma, it might be helpful to have a toy example in mind.
In the case when $\cC$ is the category of finite-dimensional vector spaces over $k$, we have $\Kzero (\cC) = \Gzero (\cC) \cong \bZ$, and then the tuple $\vec n$ is just one natural number $n \in \bN$, which corresponds to the projective object $P_n = k^{\oplus n}$.
The quot scheme $\Quot(k^{\oplus n})$ is a disjoint union $\Quot(k^{\oplus n}) = \sqcup_{\ell \geq 0} \Gr(n,\ell)$ of grassmannians, and we have $\cM_{\cC} = \sqcup_{\ell \geq 0} B\!\GL_{\ell}$.  
Then the statement of the lemma specializes to the claim that for each $\ell$, the morphism $\Gr(n,\ell) \to B\!\GL_\ell$ is quasi-affine.

\begin{proof}
Let $R$ be a $k$-algebra, and choose a morphism $\Spec(R) \to \cM_{\cC}$ corresponding to an element $V \in (\cA_{R})^{\text{fp}}$. We want to show that $\Spec(R) \times_{\cM_{\cC}} (\bigsqcup_{\vec{n}} \Quot(P_{\vec{n}})) \to \Spec(R)$ is represented by a smooth surjective morphism of schemes, with each $\Spec(R) \times_{\cM_{\cC}} \Quot(P_{\vec{n}})$ quasi-affine. Since both functors $\bigsqcup_{\vec{n}} \Quot(P_{\vec{n}})$ and $\cM_{\cC}$ commute with filtered colimits (cf. the proof of \cite[E.3.4]{artin-zhang}), it suffices to check this when $R$ is a finite type $k$-algebra.

We start by showing that $\Spec(R) \times_{\cM_{\cC}} \Quot(P_{\vec{n}})$ is represented by a quasi-affine smooth $R$-scheme for any fixed $\vec{n}$. For every $R$-algebra $S$, the set of $S$-points of $\Spec(R) \times_{\cM_{\cC}} \Quot(P_{\vec{n}})$ is the set of epimorphisms $P_{\vec{n}} \otimes S \twoheadrightarrow V \otimes_{R} S$. Consider the functor $X$ that sends an $R$-algebra $S$ to the set $\Hom_{\cA_{S}}(P_{\vec{n}} \otimes S, V \otimes_{R} S)$. We conclude the proof of quasi-affineness and smoothness by showing the following.
\begin{enumerate}[label=(\arabic*)]
    \item $X$ is represented by an affine smooth $R$-scheme.
    \item The natural monomorphism of functors $\Spec(R) \times_{\cM_{\cC}} \Quot(P_{\vec{n}}) \hookrightarrow X$ is represented by an open immersion.
\end{enumerate}
For (1), notice that $\Hom_{\cA_{R}}(P \otimes R, V)$ is a vector bundle on $\Spec(R)$ by \cref{lemma:homs from projectives families} \ref{item: Hom from fp proj preserves flatness} and \Cref{lemma: homs from fp proj preserve finiteness}. Moreover, by \cref{lemma:homs from projectives families}\ref{item: Hom from fp proj commutes with base change}, for any $R$-algebra $S$ we have $X(S) = \Hom_{\cA_{R}}(P \otimes R, V) \otimes_{R} S$. Therefore, the functor $X$ is represented by the total space of the vector bundle $\Hom_{\cA_{R}}(P \otimes R, V)$, which is affine and smooth over $R$. For (2), consider the universal homomorphism $\psi: P \otimes \cO_X \to V \otimes_{R} \cO_X$, and denote by $\cQ$ the cokernel of $\psi$. The subfunctor $\Spec(R) \times_{\cM_{\cC}} \Quot(P_{\vec{n}})$ consists of the $S$-points of $X$ such that the pullback of $\psi$ is surjective. In other words, $\Spec(R) \times_{\cM_{\cC}} \Quot(P_{\vec{n}})$ is represented by the $S$-points of $X$ such that the pullback of $\cQ$ is $0$. By Nakayama's lemma (\cref{lemma:nakayama abelian cats}), this is represented by an open subscheme of $X$.

We are left to prove surjectivity. Since $\Spec(R) \times_{\cM_{\cC}} (\bigsqcup_{\vec{n}} \Quot(P_{\vec{n}})) \to \Spec(R)$ is a smooth morphism of schemes locally of finite type over $k$, it suffices to check surjectivity on $k$-points. This amounts to showing that for every $V \in \cC$ there exists some $\vec{n}$ and a surjection $P_{\vec{n}} \twoheadrightarrow V$, which holds because $\{P_i\}_{i \in I}$ is a set of generators.
\end{proof}

\begin{prop}
The pseudofunctor $\cM_{\cC}$ is represented by an algebraic stack locally of finite type over $k$.
\end{prop}
\begin{proof}
By \cite[C8.6]{artin-zhang}, the pseudofunctor $\cM_{\cC}$ is a stack in the fppf topology. Since we have exhibited an algebraic space atlas locally of finite type over $k$ (\Cref{lemma:quot atlas}), it follows that $\cM_{\cC}$ is an algebraic stack locally of finite type over $k$ \cite[\href{https://stacks.math.columbia.edu/tag/06DB}{Tag 06DB}]{stacks-project}.
\end{proof}

\begin{defn}
For any object in $V \in \cC$, we say that $V$ has G-theory class $\alpha$ if $[V] = \alpha$ in $\Gzero(\cC)$.
\end{defn}

\begin{prop}
\label{proposition:M_alpha closed}
For every effective G-theory class $\alpha$, the subset of $k$-points of $\cM_{\cC}$ with G-theory class $\alpha$ are exactly the $k$-points of an open and closed substack $(\cM_{\cC})_{\alpha} \subset \cM_{\cC}$.
\end{prop}
\begin{proof}
Let $R$ be a finite type $k$-algebra and let $V$ be a $R$-flat noetherian object in $\cA_{R}$, corresponding to a point $\Spec(R) \to \cM_{\cC}$. Define $R(k)_{\alpha} \subset \Spec(R)(k)$ by
\[ R(k)_{\alpha} := \left\{ f \colon \Spec(k) \to \Spec(R) \; \lvert \; \text{$f^*(V)$ has G-theory class $\alpha$} \right\}
.\]

We need to show that there exists a (unique) open and closed subscheme $U \subset \Spec(R)$ with $U(k) = R(k)_{\alpha}$. For each $i \in I$, we can similarly define
\[ R(k)^i_{\alpha} := \left\{ f \colon \Spec(k) \to \Spec(R) \; \lvert \; \text{$\langle P_i, f^*(V)\rangle = \langle P_i, \alpha\rangle$} \right\}. \]
We claim that for all $i \in I$ there is a (unique) open and closed subscheme $U^i$ with $U^i(k) = R(k)^i_{\alpha}$. By \cref{lemma:homs from projectives families} \ref{item: Hom from fp proj preserves flatness} and \Cref{lemma: homs from fp proj preserve finiteness}, the module $\Hom_{\cA_{R}}(P_i, V)$ is a vector bundle on $\Spec(R)$, and by \cref{lemma:homs from projectives families} \ref{item: Hom from fp proj commutes with base change} the set $R(k)^i_{\alpha}$ consists of those closed points in $\Spec(R)(k)$ such that the fiber of the vector bundle $\Hom_{\cA_{R}}(P_i, V)$ has dimension $\langle P_i, \alpha\rangle$. This is the set of closed points of a closed and open subscheme of $\Spec(R)$, concluding the proof of the claim.
Now, we have $R(k)_{\alpha} = \bigcap_{i \in I} R(k)^i_{\alpha}$ by \Cref{lemma: dual bases}(ii). We have shown that each $R(k)^i_{\alpha}$ is the set of $k$-points of the union $U_i$ of some connected components of $\Spec(R)$. Since $R$ is of finite type over $k$, it has finitely many open and closed connected components. In particular, the intersection $R(k)_{\alpha} = \bigcap_{i \in I} R(k)^i_{\alpha}$ eventually stabilizes, thus showing that $R(k)_{\alpha}$ is the set of $k$-points of the closed and open subscheme $U_{i_1} \cap U_{i_2} \cap \ldots \cap U_{i_l}$ for some finite set of indexes $i_1, i_2, \ldots, i_l \in I$.
\end{proof}

\begin{lemma} \label{lemma:properties of the stack}
The algebraic stack $\cM_{\cC}$ satisfies the following properties.
\begin{enumerate}
    \item It has affine diagonal.
    \item It is $\Theta$-reductive and $S$-complete.
    \item It satisfies the existence part of the valuative criterion for properness.
    \item A point $V \in |\cM_{\cC}|$ is closed if and only if $V$ is semisimple.
\end{enumerate}
\end{lemma}
\begin{proof}\quad
Part (i) follows from \cite[Lem. 7.20]{alper2019existence}. Parts (ii)---(iii) follow from \cite[Lem. 7.16, 7.17, 7.18]{alper2019existence} (note that in our case we don't need to restrict to DVRs that are essentially of finite type over $k$, because $\cA$ is strongly locally noetherian). Part (iv) follows from \cite[Lem. 7.19]{alper2019existence}.
\end{proof}

\begin{prop}
\label{proposition: MCalpha is connected and qc} \label{prop:irreducible components are quasicompact}
    For every effective G-theory class $\alpha$, the open and closed substack $(\cM_\cC)_\alpha \subset \cM_\cC$ is connected and quasi-compact.
\end{prop}
\begin{proof}
    Let $A$ be the unique semisimple object of class $\alpha$.
    We will show that every object $V$ from $(\cM_\cC)_\alpha$ is in the same connected component as $A$.
    Indeed, if $0 = V_0 \subsetneq V_1 \subsetneq \ldots \subsetneq  V$ denotes the Jordan-H\"older filtration, then a choice of weights induces a morphism $\Theta_k := \mathbb{A}^1_k/\mathbb{G}_{m,k} \to (\cM_{\cC})_{\alpha}$ which sends $1$ to $[V]$ and send $0$ to the associated graded $[\bigoplus_{i=0}^{n-1} V_{i+1}/V_i]$ (see \cite[Cor. 7.13]{alper2019existence}). Since $\bigoplus_{i=0}^{n-1} V_{i+1}/V_i$ is isomorphic to the semisimple object $A$ (\cref{remark: gr V is indep of JH filtrations}), it follows that $V$ and $A$ lie on the same connected component of $(\cM_{\cC})_{\alpha}$.

    To show that $(\cM_\cC)_\alpha$ is quasi-compact, we pick any quasi-compact open substack $\cU \subset (\cM_\cC)_\alpha$ that contains the semisimple object $A$.
    We will show that $\cU = (\cM_\cC)_\alpha$ by proving that $\cU$ contains all $k$-points. Indeed, take any $k$-point of $(\cM)_{\alpha}(k)$ corresponding to an element $V$. Let $f:\Theta_k \to (\cM)_{\alpha}$ be a H\"older degeneration $f:\Theta_k \to (\cM)_{\alpha}$ as in the previous paragraph. Then the preimage $f^{-1}(\cU)$ is an open substack of $\Theta_k$ containing $0$ (since $A \in \cU(k)$). This forces $f^{-1}(\cU) = \Theta_k$, and hence we conclude that $[V] = f(1)$ belongs to $\cU$, as desired.
\end{proof}

\section{Moduli spaces}

For this section, we keep the same notation and assumptions as in \Cref{section: the stack of objects}. We assume in addition that the characteristic of $k$ is $0$.

\subsection{Good moduli space for \texorpdfstring{$\cM_{\cC}$}{MC}}

\begin{lemma} \label{lemma: irreducible components admits gms}
Let $\alpha$ be an effective G-theory class, then $(\cM_\cC)_\alpha$ admits a good moduli space $(M_\cC)_\alpha$ which is proper over $k$.
\end{lemma}
\begin{proof}
We use \cite[Thm. A]{alper2019existence}. The base field $k$ has characteristic $0$, and $(\cM_\cC)_\alpha$ is of finite type (\cref{prop:irreducible components are quasicompact}), $\Theta$-reductive, S-complete (\cref{lemma:properties of the stack} (ii)), and satisfies the existence part of the valuative criterion for properness (\cref{lemma:properties of the stack} (iii)).
\end{proof}

We denote by $M_{\cC}:= \bigsqcup_{\alpha \in \Geff(\cC)} (M_\cC)_\alpha$ the algebraic space given by the disjoint union of $(M_\cC)_\alpha$.

\begin{thm} \label{thm: good moduli space whole stack} With notation as above:
\begin{enumerate}
    \item 
    \label{item: gms for the big stack}
    $\cM_{\cC} \to M_{\cC}$ is a good moduli space.
    \item There is a bijection between the set of effective G-theory classes in $\Geff(\cC)$ and the set $M_{\cC}(k)$.
    \item For any $\alpha \in \Geff (\cC)$, the moduli space of the open and closed substack $(\cM_{\cC})_{\alpha}$ is a local finite artinian scheme over $\Spec(k)$.
\end{enumerate}  
\end{thm}
\begin{proof}
Since the statement can be checked Zariski locally on the target \cite[Prop. 4.7 (ii)]{alper-good-moduli},
it enough to recall that $\cM_\cC$ is the disjoint union of its connected components $(\cM_\cC)_\alpha$ (\cref{proposition: MCalpha is connected and qc}),
each of which admits a good moduli space by \cref{lemma: irreducible components admits gms}.

By \cite[Prop. 9.1]{alper-good-moduli} the $k$-points of $M_{\cC}$ are in bijection with closed $k$-points of the stack $\cM_{\cC}$. By \Cref{lemma:properties of the stack} (iv), the closed points of $\cM_{\cC}$ are in bijection with isomorphism classes of semisimple objects in $\cC$. Note that every effective G-theory class $[A] \in \Geff (\cC)$ has a unique semisimple representative. Therefore, we have a bijection between the effective G-theory classes in $\Geff (\cC)$ and the set $M_{\cC}(k)$. This also implies that the moduli space of each open and closed substack $(\cM_{\cC})_{\alpha}$ has a single $k$-point, and so must be a local finite artinian scheme over $\Spec(k)$.
\end{proof}

\begin{remark}
Since the formation of moduli spaces commutes with base-change \cite[Prop. 4.7(i)]{alper-good-moduli}, it follows from our description of $M_{\cC}$ that for any over-field $K \supset k$ the base-change functor $(-)\otimes K$ induces a canonical identification of G-theory groups $\Gzero (\cC) = \Gzero ((\cA_K)\rfp)$.
\end{remark}

\begin{example}
    When $\cA$ is the category of representations of an acyclic quiver, then each $(M_\cA)_\alpha \cong \Spec k$, see e.g. \cite{bdfhmt}. However, this will not be true in general. For example, take $\cA = k[\varepsilon]\dash\Mod$, where $\varepsilon ^2 = 0$. Then $\Gzero(\cA\rfp) \cong \bZ$, where $[M]$ is identified with $\dim M \in \bZ$. Taking $\alpha = 1 \in \bZ$, we get $(\cM_\cA)_\alpha \cong \Spec k[\varepsilon] \times B\Gm$, and so $(M_\cA)_\alpha \cong \Spec k[\varepsilon]$.
\end{example}

\begin{remark}\label{remark 4.5}
    In view of \Cref{thm: good moduli space whole stack}, one can apply \cite[Thm. 7.4.10]{andresIN_thesis} to conclude that for all $\alpha \in \Geff$ the moduli stack $\cM_{\alpha}$ embeds as a closed substack of a stack of representations of a quiver. We note that the conclusion that the stack is linearly lit from \cite[Thm. 7.4.10]{andresIN_thesis} does not technically guarantee the embedding if the good moduli space is not reduced, but the proof of the necessary result \cite[Prop. 7.4.1]{andresIN_thesis} for obtaining the embedding also applies in the case when the good moduli space is the spectrum of a local artinian algebra.
    
    In particular, $\cM_{\alpha}$ is isomorphic to the quotient stack of an affine scheme by an action of a product of general linear groups.
    We would like to stress that this GIT presentation of the stack $\cM_{\alpha}$ only becomes apparent a posteriori, after one proves that $\cM_{\alpha}$ admits a good moduli space which is a point (possibly nonreduced).
    Even with this quotient presentation in mind, we believe that it is conceptually clearer to approach the construction of stability conditions on $\cM_{\alpha}$ from the intrinsic point of view, as we do in the next subsection.
\end{remark}

\subsection{Stability of objects in \texorpdfstring{$\cC$}{C}} \label{section: stability of objects in C}
Fix an effective G-theory class $\alpha \in \Geff (\cC)$. In this subsection we define a family of stability conditions coming from numerical invariants on $(\cM_{\cC})_{\alpha}$ in the sense of \cite{halpernleistner2018structure}. We first need to define some line bundles on $(\cM_{\cC})_{\alpha}$.
\begin{defn}
Let $\beta \in \Kzero (\cC) \cong \mathbb{Z}^{\oplus I}$ of the form $\sum_{i \in I} n_i [P_i]$, for some tuple of integers $n_i$. We define $\cL_{\beta}$ to be the line bundle in $\text{Pic}(\cM_{\cC})$ such that for every noetherian $k$-algebra $R$ and every $f \colon \Spec(R) \to \cM_{\cC}$ determined by an $R$-flat element $A \in (\cA_{R})\rfp$, we have
\[
    f^*(\cL_{\beta}) := \bigotimes_{i \in I}\text{det}(\Hom_{\cA_{R}}(P_i \otimes R, A))^{\otimes n_i}.
\]
\end{defn}

The definition above makes sense, because each $\Hom_{\cA_{R}}(P_i \otimes R, A)$ is a vector bundle on $\Spec(R)$ (Lemmas \ref{lemma:homs from projectives families} and \ref{lemma: homs from fp proj preserve finiteness}), and $n_i =0$ for all but finitely many $i \in I$.

\begin{remark}
    Let $E = \bigoplus_{w \in \mathbb{Z}} E^w$ be a graded object in $\cC$, which we may equivalently think of as a morphism $g: B(\mathbb{G}_{m})_k \to \cM_{\cC}$. Then, the $\mathbb{G}_m$-weight of the pullback is given by $\wt(g^*(\cL_{\beta})) = \sum_{w\in \mathbb{Z}}w\langle \beta,E^w\rangle$.
\end{remark}

\begin{defn}
Let $\gamma = \sum_{i \in I} m_i [P_i] \in \Kzero (\cC) \otimes_{\mathbb{Z}}\mathbb{Q}$ be a rational K-theory class such that $m_i \geq 0$ for all $i \in I$. For any field $K$ and any $V \in (\cA_{K})^{\text{fp}}$, we define the $\gamma$-length of $V$ by 
\[\ell_{\gamma}(V) := \left \langle \gamma_K,  V \right \rangle \geq 0,\]
where $\gamma_K$ denotes the base-change $\sum_{i \in I} m_i [P_i \otimes K] \in \Kzero((\cA_K)\rfp) \otimes_{\mathbb{Z}}\mathbb{Q}$. Similarly for a G-theory class $\alpha$, we set $\ell_{\gamma}(\alpha) := \langle \gamma, \alpha \rangle$.
\end{defn}

\begin{defn}
    Let $\alpha$ be a G-theory class. Let $J_{\alpha} \subset I$ be the set of all $j \in I$ such that $\langle P_j, \alpha\rangle \neq 0$ (equivalently, we have $\alpha = \sum_{i \in J_{\alpha}} a_i [S_i]$ for some $a_i >0$). We say that a class $\gamma = \sum_{i \in I} m_i [P_i] \in \Kzero (\cC)$ is $\alpha$-nondegenerate if $m_i >0$ for all $i \in J_{\alpha}$.
\end{defn}

If $\gamma$ is an $\alpha$-nondegenerate class, then for any $B \in \cC$ with G-theory class $\alpha$ and any $0 \neq E \subset B$, we have $\ell_{\gamma}(E) >0$.

\begin{example} \label{example: simple length gamma}
Let $\gamma = \sum_{i \in J_{\alpha}} \frac{1}{\langle P_i, S_i \rangle} [P_i]$. Then, for all $B \in \cC$ with G-theory class $\alpha$ and all $E \subset B$, the $\gamma$-length $\ell_{\gamma}(E)$ is just the length of $E$ in the abelian category $\cC$ (i.e. the size of its Jordan-H\"older filtration).
\end{example}

\begin{defn}
 Let $\beta \in K_0(\cC)$, and let $\gamma \in \Kzero (\cC) \otimes_{\mathbb{Z}} \mathbb{Q}$ be an $\alpha$-nondegenerate class. Then  $\cL_{\beta, \gamma}$ is defined to be the line bundle on $(\cM_{\cC})_{\alpha}$ given by
 \[ \cL_{\beta, \gamma} :=   \cL_{\beta}^{\otimes \ell_{\gamma}(\alpha)} \otimes
 \cL_{\gamma}^{\otimes -\langle \beta, \alpha\rangle} .\]
\end{defn}

For the rest of this section, we fix $\beta$ and $\gamma$ as above. 

\begin{defn} \label{defn: beta-semistability}
For any object $E \in \cC$ with $\ell_{\gamma}(E) \neq 0$, we define the $\beta$-slope of $E$ by $\sigma_{\beta}(E):=\frac{\langle\beta,E\rangle}{\ell_{\gamma}(E)}$. An object $A\in \cC$ with numerical class $\alpha$ is called \emph{$\beta$-semistable} if for all nonzero subobjects $0 \neq E \subset A$, we have the inequality of slopes $\sigma_{\beta}(E) \leq \sigma_{\beta}(A)$.
\end{defn}
Our goal is to show that the set of $\beta$-semistable objects with G-theory class $\alpha$ are the $k$-points of an open substack $(\cM_{\cC})_{\alpha}^{\beta\dash ss} \subset (\cM_{\cC})_{\alpha}$ which admits a projective good moduli space over $k$. We would also like to develop Harder-Narasimhan theory for this stability condition. We will accomplish this via the use of a numerical invariant in the sense of \cite[Def. 0.0.3]{halpernleistner2018structure}.

For any field $K \supset k$ and any positive integer $h$, morphisms $B(\mathbb{G}_m)^h_{K} \to \cM_{\cC}$ amount to $\bZ^h$-graded objects $\bigoplus_{\vec{w} \in \bZ^h} A_{\vec{w}}$ in $(\cA_{K})^{\text{fp}}$ (cf. \cite[Prop. 7.12]{alper2019existence}). Using this description, we define a rational quadratic norm on graded points as in \cite[Def. 4.1.12]{halpernleistner2018structure}.
\begin{defn}
We denote by $b_{\gamma}$ the rational quadratic norm on graded points of $(\cM_{\cC})_{\alpha}$ defined as follows. For any $h>0$, morphism $g \colon B(\mathbb{G}_m)^h_{K} \to (\cM_{\cC})_{\alpha}$ corresponding to a graded object $\bigoplus_{\vec{w} \in \bZ^h} A_{\vec{w}}$, and vector $\vec{r} \in \mathbb{R}^h$, we set
\[ b_{\gamma}(g)(\vec{r}) = \sum_{\vec{w} \in \bZ^h} 
\left(\vec{w} \cdot \vec{r}
 \right)^2  
\ell_{\gamma}
\left( A_{\vec{w}} \right). \]
where $\vec{w} \cdot \vec{r} = \sum_{j=1}^h w_j r_j$ denotes the standard inner product for two vectors $\vec{w} = (w_j)_{j=1}^h$ and $\vec{r} = (r_j)_{j=1}^h$.
\end{defn}

\begin{defn}
We define $\mu_{\beta}$ to be the $\mathbb{R}$-valued numerical invariant on the stack given by $\mu_{\beta} = \text{wt}(\cL_{\beta, \gamma})/\sqrt{b_{\gamma}}$
\end{defn}

We recall that under certain conditions, such numerical invariant can be used to define a $\Theta$-stratification on the stack (cf. \cite[\S4.1]{halpernleistner2018structure} or \cite[\S2.5]{torsion-freepaper}).
\begin{thm} \label{thm: theta stratification and moduli space stability condition} Fix $\beta \in \Kzero(\cC)$ and $\gamma \in \Kzero(\cC)\otimes_{\mathbb{Z}}\mathbb{Q}$ that is $\alpha$-nondegenerate.
\begin{enumerate}
    \item The numerical invariant $\mu_{\beta}$ defines a $\Theta$-stratification on the stack $(\cM_{\cC})_{\alpha}$.
    \item The open $\mu_{\beta}$-semistable locus $(\cM_{\cC})^{\mu_{\beta}\dash ss}_{\alpha}$ admits a good moduli space $(M_{\cC})_{\alpha}^{\beta \dash ss}$ that is projective over $\Spec(k)$. A power of the dual line bundle $\cL_{\beta, \gamma}^{-1}$ descends to an ample line bundle on $(M_{\cC})_{\alpha}^{\beta \dash ss}$.
    \item The $k$-points of the $\mu_{\beta}$-semistable open locus $(\cM_{\cC})_{\alpha}^{\mu_{\beta}\dash ss}$ are exactly the $\beta$-semistable objects.
\end{enumerate} 
\end{thm}
\begin{proof}
Let $\varphi: (\cM_{\cC})_{\alpha} \to (M_{\cC})_{\alpha}$ be the separated good moduli space of the open and closed substack $(\cM_{\cC})_{\alpha} \subset \cM_{\cC}$, which exists by \Cref{thm: good moduli space whole stack}. Recall that the good moduli space $(M_{\cC})_{\alpha}$ is finite over $\Spec(k)$. To conclude (i) and (ii), we use \cite[Thm. 5.6.1 (1)+(2)]{halpernleistner2018structure} with $\mathfrak{X} = \mathfrak{Y} = (\cM_{\cC})_{\alpha}$ and $\mathcal{L} = \mathcal{L}^{-1}_{\beta, \gamma}$ and $\|\bullet \| = b_{\gamma}$. In this case, the weak $\Theta$-stratification will automatically be a $\Theta$-stratification because the characteristic of $k$ is $0$.

We are left to prove (iii). Let $p \in (\cM_{\cC})_{\alpha}(k)$ corresponding to an object $A \in \cC$ with G-theory class $\alpha$. By \cite[Cor. 7.13]{alper2019existence}, $\Theta$-filtrations of $p$ correspond to $\mathbb{Z}$-weighted filtrations $(A_w)_{w \in \mathbb{Z}}$, where 
\begin{itemize}
    \item $A_w\subset A_{w-1} \subset A$ for all $w \in \mathbb{Z}$.
    \item $A_w=0$ for $w \gg 0$, and $A_w = A$ for $w \ll 0$.
\end{itemize}
For any such filtration $f$ corresponding to $(A_w)_{w \in \mathbb{Z}}$, a direct computation shows that the value of the numerical invariant is
\[ \mu_{\beta}(f) = \frac{\sum_{w \in \mathbb{Z}} w \left(\ell_{\gamma}(\alpha)\cdot \langle \beta, A_w/A_{w+1}\rangle - \langle \beta, \alpha\rangle \cdot \ell_{\gamma}(A_w/A_{w+1})\right)}{\sqrt{\sum_{w \in \mathbb{Z}} w^2 \ell_{\gamma}(A_w/A_{w+1})}}.\]
Therefore the point $p = A$ is $\mu_{\beta}$-semistable if and only if for all weighted filtrations $(A_w)_{w \in \mathbb{Z}}$ we have
\[ \sum_{w \in \mathbb{Z}} w (\ell_{\gamma}(\alpha)\cdot \langle \beta, A_w/A_{w+1} \rangle - \langle \beta, \alpha \rangle \cdot \ell_{\gamma}(A_w/A_{w+1})) \leq 0. \]
Using the additivity relations $\langle \beta, A_{w}/A_{w+1}\rangle=\langle \beta,A_{\omega}\rangle-\langle \beta,A_{w+1}\rangle$ and $\ell_{\gamma}(A_w/A_{w+1}) = \ell_{\gamma}(A_w) - \ell_{\gamma}(A_{w+1})$ as well as the identities $\langle \beta, A \rangle = \langle \beta, \alpha \rangle$ and $\ell_{\gamma}(A) = \ell_{\gamma}(\alpha)$, we rewrite this as follows
\[  \sum_{w \in \mathbb{Z}} (\ell_{\gamma}(A) \langle \beta, A_w \rangle - \ell_{\gamma}(A_w) \langle \beta, A \rangle ) \leq 0. \]
A standard argument (cf. \cite[Prop. 5.16]{torsion-freepaper}) shows that the inequality above is satisfied for all filtrations $(A_w)_{w \in \mathbb{Z}}$ if and only if $A$ is $\beta$-semistable in the sense of \Cref{defn: beta-semistability}.
\end{proof}

\section{Examples} \label{section: example}

\subsection{Modules over a finite dimensional associative algebra}
\label{subsection: fin dim assoc algebras}

\begin{prop} \label{prop: cat of representations of algebra satisfies conditons}
Let $A$ be an associative $k$-algebra that is finite-dimensional as a $k$-vector space. Then the category $\cC = A\dash \fdmod$ of finite dimensional left modules over $A$ is nearly finite in the sense of \Cref{standing assumptions on C}.
\end{prop}
\begin{proof}
All objects of $\cC$ have finite dimension as $k$-vector spaces, and so they are of finite length. It also follows that $\cC$ is Hom-finite. Since $A$ regarded as an $A$-module is a projective generator, $\cC$ satisfies \ref{standing assumption: C has enough projectives}. Moreover, every simple element of $\cC$ is a quotient module of $A$, and so the isomorphism classes of simple objects in $\cC$ form a set. \qedhere
\end{proof}

In view of \Cref{prop: cat of representations of algebra satisfies conditons}, our main results in \Cref{thm: theta stratification and moduli space stability condition} apply to the moduli of representations of a finite-dimensional associative $k$-algebra. This moduli problem and stability conditions have been considered by King \cite[\S 4]{king}. For comparison, we note that King is not completely explicit about the nonreduced structure of the moduli problem, opting to work with varieties instead, whereas the Artin-Zhang functor equips the moduli spaces with a natural (possibly nonreduced) scheme structure. This can, of course, also be made explicit from the point of view in \cite{king}.

\begin{example}[Representations of acyclic quivers]
For the following example, we refer the reader to \cite{king} and \cite{bdfhmt} and references therein. Let $Q$ be an acyclic quiver. The category $\cC$ of finite dimensional representations of $Q$ is isomorphic to finite dimensional representations of its path algebra, which is an associative algebra of finite dimension over $k$. In particular, $\cC$ satisfies \Cref{standing assumptions on C} by \Cref{prop: cat of representations of algebra satisfies conditons}.

In this case we have $\Gzero (\cC) = \mathbb{Z}^{\oplus I}$, where $I$ is the set of vertices of the quiver $Q$. 
For each vertex $i \in I$, there is an associated projective object $P_i$ in $\cC$. For any representation $V$ and $i \in I$, the pairing $\langle P_i,V \rangle$ is the dimension of the underlying $k$-vector space $V_i$ at the vertex $i$. 
Given a class $\beta = \sum_{i \in I} n_i[P_i] \in \Kzero(\cC)$ and $\gamma$ as in \Cref{example: simple length gamma}, our definition of $\beta$-semistability (\cref{defn: beta-semistability}) agrees with the one defined by King \cite{king}.
Hence, we recover the moduli spaces of representations of $Q$ constructed using GIT in \cite{king} and intrinsically in \cite{bdfhmt}.
\end{example}

We offer another quiver example that is related to representations of finite-dimensional algebras, even though it does not strictly fall into that framework.
\begin{example}[Representations of semi-infinite quivers]
    Let $Q$ be a quiver with possibly infinitely many vertices and edges. We say that $Q$ is \emph{projectively semi-infinite} if for any vertex $i$, the number of paths starting in $i$ is finite. Equivalently, we require that any indecomposable projective representation $P$ has finite total dimension $\sum_{i \in Q_0} \dim P_i$.
    
    Set $\cC = \rep Q$ to be the category of representations of $Q$ whose total dimension is finite. Then $\cC$ is nearly finite.
    In this case, the connected components of $\cM_{\cC}$ correspond to the moduli stacks of representation of finite acyclic subquivers $W \subset Q$ with fixed dimension vector.
\end{example}

\subsection{Category \texorpdfstring{$\cO$}{O}}

Let $\mathfrak{g}$ be a complex semisimple Lie algebra. Fix a Cartan subalgebra $\mathfrak{h}$ and let $\mathfrak{h}^{\vee}$ be its dual. Fix a system of simple roots $\Phi$ and let $\Phi^{+}$ be the positive roots. Let $\mathfrak{n}=\oplus_{\alpha>0}\mathfrak{g}_{\alpha}$.

\begin{defn}
    The \emph{BGG category} $\mathcal{O}$ is the full subcategory of $\Mod U(\mathfrak{g})$ such that
    \begin{enumerate}
        \item $M$ is a finitely generated $U(\mathfrak{g})$-module.
        \item $M$ is $\mathfrak{h}$-semisimple, i.e. $M=\oplus_{\lambda\in \mathfrak{h}^{\vee}}M_{\lambda}$.
        \item $M$ is locally $\mathfrak{n}$-finite, i.e. $\forall v\in M, U(\mathfrak{n})\cdot v$ is finite dimensional.
    \end{enumerate}
\end{defn}

\begin{thm}
Category $\mathcal{O}$ is a nearly finite abelian category.
\end{thm}
\begin{proof}
Since $\mathcal{O}$ is noetherian, artinian and Hom-finite by \cite[Thm. 1.1, 1.11]{humphreys.O}, 
\ref{standing assumption: C Hom-finite} and \ref{standing assumption: C finite} are satisfied. It is essentially small because simple objects are parametrized by linear functionals in $\mathfrak{h}^{\vee}$ by \cite[Thm. 1.3]{humphreys.O}. Furthermore, \ref{standing assumption: C has enough projectives} holds by \cite[Thm. 3.8]{humphreys.O}.
\end{proof}

Let $M(\lambda)$ denote the Verma module associated to $\lambda\in \mathfrak{h}^{\vee}$. Its unique simple quotient $L(\lambda)$ is the unique simple object of highest weight $\lambda$ \cite[Thm. 1.3]{humphreys.O}. Each $L(\lambda)$ has a unique indecomposable projective cover $P(\lambda)$ \cite[Thm. 3.9]{humphreys.O}. In this case, every object is perfect and hence the $G$-theory and $K$-theory are the same. For any $M\in \mathcal{O}$, its $K$-theory class is given by $\sum_{\lambda} [M: L(\lambda)][L(\lambda)]$ where $[M:L(\lambda)]$ is composition factor multiplicity.

\begin{example}
Consider $\mathfrak{g}=\mathfrak{sl}(2,\mathbb{C})$. Then for any block $\mathcal{O}_{\chi_{\lambda}}$, there are five non-isomorphic indecomposable modules
$$L(\lambda), L(-2-\lambda)=M(-2-\lambda), M(\lambda)=P(\lambda), M(\lambda)^{\vee}, P(-2-\lambda).$$
For a fixed class $\alpha$, then there are only finitely many objects with that class. Hence, our moduli spaces would be a finite number of points.

Consider the principal block $\mathcal{O}_{\chi_0}$, that is $\lambda=0$. Fix a class $\beta=\beta_1[L(0)]+\beta_2[L(-2)]$ and take $\gamma$ as in Example \ref{example: simple length gamma}. Then we find that the $\beta$-slope $\sigma_{\beta}$ is
$$\sigma_{\beta}(L(0))=\beta_1, \sigma_{\beta}(L(-2))=\beta_2$$
$$\sigma_{\beta}(M(0))=\sigma_{\beta}(M(0)^{\vee})=\frac{\beta_1+\beta_2}{2}$$
$$\sigma_{\beta}(P(-2))=\frac{\beta_1+2\beta_2}{3}.$$
It follows that $M(0)$ is $\beta$-semistable if and only if $\beta_1\geq \beta_2$. Also, $M(0)^{\vee}$ is $\beta$-semistable if and only if $\beta_2\geq \beta_2$. Meanwhile, $P(-2)$ is $\beta$-semistable only when $\beta_1=\beta_2$, which gives the trivial stability condition. We have recovered the stability conditions given in \cite[Ex. 3.4]{futorny2007moduli}.

\end{example}

\subsection{Comodules over co-Frobenius Hopf algebras}
\label{subsection: co-Frobenius}

Let $C$ be a counital coalgebra. 
Set $\cC = H\dash\fdcomod$, so $\cA = H\dash\Comod = \Ind (\cC)$ by the fundamental theorem of coalgebra.
The category $\cC$ automatically satisfies \ref{standing assumption: C Hom-finite} and \ref{standing assumption: C finite}.
For the fact that $\cC$ is essentially small, we use that simple comodules are in bijection with simple subcoalgebras of $H$ \cite[pg. 354]{larson-hopf-algebras}.

\begin{thm}
[{\cite[Thms 3, 10]{bertrand-semiperfect-coalgebras}}]
\label{theorem: equivalent defs of co-Frobenius Hopf algebras}
    The following are equivalent for a Hopf algebra $H$.
    \begin{enumerate}
        \item $H$ is co-Frobenius.
        \item The injective hull of a simple $H$-comodule is finite-dimensional.
    \end{enumerate}
\end{thm}

For the purpose of the present paper, we use this theorem as a definition of a co-Frobenius Hopf algebra.

\begin{coroll}
    Let $H$ be a co-Frobenius Hopf algebra.
    Then $\cC = H\dash\fdcomod$ is nearly finite.
\end{coroll}

\begin{proof}
    It remains to check that $\cC$ has enough projective objects \ref{standing assumption: C has enough projectives}. For this, we show that every simple object $S \in \cC$ admits a surjection from a projective object $P \to S$. Since $H$ is a Hopf algebra, using its antipode ensures that taking the $k$-linear dual $(\blank)^\vee \colon \cC\rop \to \cC$ is a self-antiequivalence of $\cC$. The comodule $S^\vee$ is still simple, so by \cref{theorem: equivalent defs of co-Frobenius Hopf algebras}, $S^\vee \in \cC$ admits a finite-dimensional injective hull $S^\vee \to E$, so $E \in \cC$. The comodule $P = E^\vee$ is projective, since an antiequivalence sends injective objects to projective.
    Hence, we get the desired surjection $P \to S$.
\end{proof}

We next present an example of a co-Frobenius Hopf algebra which has infinitely many simple comodules. 
This is relevant because, if there were only finitely many, the category of comodules would be equivalent to the category of representations of a finite-dimensional associative algebra, but moduli spaces have been known for them since 1994 \cite{king}, see \cref{subsection: fin dim assoc algebras} for more details.

\begin{example}[Finite dimensional comodules over quantum groups]
    An example known to representation theorists includes the algebra $\cO_q(\mathfrak{g}) \subset U_q(\mathfrak{g})^\vee$ of functions on Lusztig's quantum group at a root of unity $q$ attached to a simple Lie algebra $\mathfrak{g}$.
    The Hopf algebra $\cO_q(\mathfrak{g})$ was proved to be co-Frobenius in
    \cite[Example 4.1]{AD-qgroups-are-coFrob}.
    Let $\cC$ denote the category of finite-dimensional comodules over $\cO_q(\mathfrak{g})$. We claim that the set of simple objects in $\cC$ is infinite. 
    Indeed, every simple comodule is a subcomodule of $\cO_q(\mathfrak{g})$. By the fundamental theorem of coalgebra, every element of a comodule is contained in a finite-dimensional subcomodule; in particular, every simple comodule is finite-dimensional, and every element of $\cO_q(\mathfrak{g})$ is contained in a simple subcomodule.
    But $\cO_q(\mathfrak{g})$ is infinite-dimensional, hence there will necessarily be infinitely many simple comodules. 
\end{example}

\footnotesize{\bibliography{moduli_abelian.bib}}
\bibliographystyle{alpha}
\end{document}